\documentclass[12pt,reqno,draft]{article} 
\usepackage{amsmath,amssymb,amsthm,amsfonts, indentfirst}
\usepackage{enumerate,color,bm}
\usepackage{mathtools}
\makeatletter
\def\@cite#1#2{[{{\bfseries #1}\if@tempswa , #2\fi}]}
\renewcommand{\section}{%
\@startsection{section}{1}{\z@}
{0.5truecm plus -1ex minus -.2ex}%
{1.0ex plus .2ex}{\bfseries\large}}
\def\@seccntformat#1{\csname the#1\endcsname.\ }
\makeatother

\numberwithin{equation}{section} 
\theoremstyle{theorem}
\newtheorem{thm}{Theorem}[section]

\newtheorem{lem}[thm]{Lemma}

\theoremstyle{definition}

\newtheorem{remark}{Remark}[section]
\newtheorem*{prthstab}{Proof of Theorem~\ref{stab}}

\newcommand{\ep}{\varepsilon}
\newcommand{\pa}{\partial}
\newcommand{\R}{\mathbb{R}}
\newcommand{\N}{\mathbb{N}}
\newcommand{\cl}[1]{{\overline#1}}

\newcommand{\baru}{\overline{u_0}}

\begin{document}
\footnote[0]{
    2020{\it Mathematics Subject Classification}\/. 
    Primary: 35B40; 
    Secondary: 35K59, 92C17.
    }
\begin{center} 
    \Large{{\bf 
    Stabilization for small mass \\in a quasilinear 
    parabolic--elliptic--elliptic \\
    attraction-repulsion chemotaxis system 
    with\\ density-dependent 
    sensitivity: \\
    repulsion-dominant case
    }}%
\end{center}
\vspace{5pt}
\begin{center}
    Yutaro Chiyo 
    \footnote[0]{
    E-mail: 
    {\tt ycnewssz@gmail.com} 
    }\\
    \vspace{12pt}
    Department of Mathematics, 
    Tokyo University of Science\\
    1-3, Kagurazaka, Shinjuku-ku, 
    Tokyo 162-8601, Japan\\
    \vspace{2pt}
\end{center}
\begin{center}    
    \small \today
\end{center}

\vspace{2pt}
\newenvironment{summary}
{\vspace{.5\baselineskip}\begin{list}{}{%
     \setlength{\baselineskip}{0.85\baselineskip}
     \setlength{\topsep}{0pt}
     \setlength{\leftmargin}{12mm}
     \setlength{\rightmargin}{12mm}
     \setlength{\listparindent}{0mm}
     \setlength{\itemindent}{\listparindent}
     \setlength{\parsep}{0pt}
     \item\relax}}{\end{list}\vspace{.5\baselineskip}}
\begin{summary}
{\footnotesize {\bf Abstract.} 
This paper deals with the quasilinear 
attraction-repulsion chemotaxis system
%
\begin{align*}
    \begin{cases}
        u_t=\nabla\cdot \big((u+1)^{m-1}\nabla u
            -\chi u(u+1)^{p-2}\nabla v
            +\xi u(u+1)^{q-2}\nabla w\big), 
          \\
        0=\Delta v+\alpha u-\beta v,
          \\
        0=\Delta w+\gamma u-\delta w
    \end{cases}
\end{align*}
%
in a bounded domain $\Omega \subset \mathbb{R}^n$ $(n \in \mathbb{N})$ 
with smooth boundary $\partial \Omega$, 
where $m, p, q \in \R$, $\chi, \xi, \alpha, \beta, \gamma, \delta>0$ 
are constants. 
In the case that $m=1$ and $p=q=2$, 
when $\chi\alpha-\xi\gamma<0$ and $\beta=\delta$, 
Tao--Wang (Math.\ Models Methods Appl.\ Sci.; 2013; 23; 1--36) 
proved that global bounded classical solutions toward 
the spatially constant equilibrium 
$(\overline{u_0}, \frac{\alpha}{\beta}\overline{u_0}, 
\frac{\gamma}{\delta}\overline{u_0})$
via the reduction to the Keller--Segel system 
by using the transformation $z:=\chi v-\xi w$, 
where $\overline{u_0}$ is the spatial average of the initial data $u_0$. 
However, since the above system involves nonlinearities, 
the method is no longer valid. 
The purpose of this paper is to establish that 
global bounded classical solutions 
converge to the spatially constant equilibrium $(\overline{u_0}, \frac{\alpha}{\beta}\overline{u_0}, 
\frac{\gamma}{\delta}\overline{u_0})$. 
}
\end{summary}
\vspace{10pt}

\newpage

\section{Introduction} \label{Sec1}


In this paper we study stabilization that solutions converge to 
a spatially constant equilibrium in 
the quasilinear attraction-repulsion chemotaxis system
%
\begin{align*}
    \begin{cases}
        u_t=\nabla\cdot \big((u+1)^{m-1}\nabla u
            -\chi u(u+1)^{p-2}\nabla v
            +\xi u(u+1)^{q-2}\nabla w\big), 
          \\
        0=\Delta v+\alpha u-\beta v,
          \\
        0=\Delta w+\gamma u-\delta w, 
    \end{cases}
\end{align*}
%
where $m, p, q \in \R$, $\chi, \xi, \alpha, \beta, \gamma, \delta>0$
are constants, 
and the functions $u$, $v$ and $w$ show the cell (or organism) density, 
the concentrations of the chemoattractant and chemorepellent, respectively. 
This system is a generalization of the chemotaxis models 
introduced by Keller--Segel~\cite{KS-1970}, 
where 
$(u+1)^{m-1}$, appearing in the diffusion term, 
means that the cell density increases and so does 
the diffusion rate. 
Also, the quasilinear sensitivities 
$u(u+1)^{p-2}$ and $u(u+1)^{q-2}$, 
which denote the density-dependent probabilities for cells 
to find space somewhere in their neighboring locations, 
were initially proposed by Painter--Hillen~\cite{PH-2002} 
via the approach of assuming the presence of a volume-filling effect and 
were studied by Tao--Winkler~\cite{TW-2012}. 
Here {\it chemotaxis} is the property of cells to move 
in a directional manner in response to concentration gradients 
of chemical substances. 
For instance, in bacteria such as E. coli, 
this property causes them to move toward the chemoattractant 
and away from the chemorepellent. 
One of the models describing such a chemotactic process is 
a fully parabolic attraction-repulsion chemotaxis system 
introduced by
Painter--Hillen~\cite{PH-2002} to idealize the quorum effect 
in the chemotactic process and Luca et al.\ \cite{LCEM-2003} 
to show the aggregation of microglia observed in Alzheimer's disease. 

\medskip

The original Keller--Segel system proposed in \cite{KS-1970} is as follows: 
%
\begin{align*}
    \begin{cases}
        u_t=\nabla \cdot\big(\nabla u
            -\chi u\nabla v\big), 
          \\
        v_t=\Delta v-v+u.
    \end{cases}
\end{align*}
%
After that, many versions have been derived from 
this system (see Hillen--Painter~\cite{HP-2009}) 
and have been 
extensively studied (see e.g., Bellomo et al.\ \cite{BBTW-2015}, 
Arumugam--Tyagi~\cite{AT-2021}). 
In particular, one of which is the following quasilinear version: 
%
\begin{align*}
    \begin{cases}
        u_t=\nabla\cdot \big((u+1)^{m-1}\nabla u
            -\chi u(u+1)^{p-2}\nabla v \big), 
          \\
        v_t=\Delta v-v+u.
    \end{cases}
\end{align*}
%
This system, which was proposed in \cite{PH-2002}, 
has been well investigated. 
For example, finite-time blow-up was shown in 
\cite{CS-2012, CS-2014, W-2013, W-2018}; 
boundedness was proved in 
\cite{ISY-2014, IY-2020, LX-2015, SK-2006, TW-2012}. 
More precisely, Winkler~\cite{W-2013} and 
Cie\'slak--Stinner~\cite{CS-2012, CS-2014} established
finite-time blow-up under the condition $p > m + \frac{2}{n}$, where $n$ denotes the spatial dimension. 
On the other hand, Tao--Winkler~\cite{TW-2012} derived boundedness 
in the above system on a convex domain when $p < m + \frac{2}{n}$; 
after that, the convexity of the domain 
was removed by Ishida--Seki--Yokota~\cite{ISY-2014}. 
Also, Ishida--Yokota~\cite{IY-2020} guaranteed 
existence 
of global bounded weak solutions in
the above system 
on the whole space. 
Besides, in the critical case $p =m + \frac{2}{n}$, 
blow-up and boundedness were classified by the
condition for the initial data $u_0, v_0$ (see e.g., 
\cite{BL-2013, IY-2012, LM-2017, M-2017}). 
As to stabilization in the above quasilinear system, 
in three or more space dimensions 
Winkler~\cite{W-2010-1} proved that solutions converge to 
a constant steady state at an exponential rate as 
$t \to\infty$ under 
smallness conditions
for $\|u_0\|_{L^\sigma(\Omega)}$ 
and $\|\nabla v_0\|_{L^\theta(\Omega)}$ for all $\sigma>\frac{n}{2}$ and all 
$\theta>n$, where $\Omega$ is a bounded domain; after the work, Cao~\cite{Cao-2015} extended the result 
in the critical case that $\sigma=\frac{n}{2}$ and $\theta=n$ 
in two or more space dimensions. 
Also, in two or more space dimensions 
Cie\'slak--Winkler~\cite{CW-2017} showed that 
a global bounded classical solution exists 
and that the solution approaches the spatially constant equilibrium 
$(\baru, \baru)$, where $\baru:=\frac{1}{|\Omega|}\int_\Omega u_0$,  
in two or more space dimensions 
under the conditions that $p-m \in [0, \frac{2}{n})$ 
and that 
$\chi\|u_0\|_{L^1(\Omega)}^{p-m} \le 1/C_{\langle p-m\rangle}^2$, 
where $C_{\langle p-m\rangle}>0$ is a constant appearing 
in the Poincar\'e--Sobolev inequality. 
Some related works for 
the corresponding quasilinear chemotaxis system of
parabolic--elliptic type, 
which shows that comparatively fast diffusion 
of the respective chemical substances, 
can be found in \cite{N-1995, N-2001, TW-2007}. 
Also, Ishida~\cite{I-2013} gave the $L^\infty$-decay property 
in the super-critical case with small initial data. 
Moreover, Mizukami~\cite{Mizu-2018, Mizu-2019} 
built bridge 
between the parabolic--parabolic Keller--Segel system and 
the parabolic--elliptic version. 
\medskip

\noindent
{\bf A question in stabilization.}
%
Let us start with a review of known results on stabilization in  
the attraction-repulsion chemotaxis system
%
\begin{align*}
    \begin{cases}
        u_t=\nabla\cdot \big(\nabla u
            -\chi u\nabla v
            +\xi u\nabla w\big), 
          \\
        0=\Delta v+\alpha u-\beta v,
          \\
        0=\Delta w+\gamma u-\delta w. 
    \end{cases}
\end{align*}
In the literatures \cite{LLM-2015, LMW-2015, TW-2013}, 
it was established that global bounded classical solutions 
approach a spatially homogeneous steady state 
by the reduction to the Keller--Segel system via 
the transformation $z := \chi v-\xi w$. 
More precisely, Tao--Wang~\cite{TW-2013} obtained that 
the solution towards the spatially constant equilibrium 
$(\baru, \frac{\alpha}{\beta}\baru, \frac{\gamma}{\delta}\baru)$
under the conditions that $\chi\alpha-\xi\gamma<0$ and that $\beta=\delta$. 
Thereafter, Li et al.~\cite{LLM-2015} and Lin et al.~\cite{LMW-2015} 
derived asymptotic behavior 
by supposing some smallness condition for $u_0$ 
instead of the condition $\beta=\delta$ 
in the parabolic--elliptic--elliptic case (\cite{LLM-2015}) 
and the parabolic--parabolic--parabolic case (\cite{LMW-2015}). 
However, for the system in which the first equation of the above one 
is replaced by 
\begin{align*}
u_t=\nabla\cdot \big((u+1)^{m-1}\nabla u
            -\chi u(u+1)^{p-2}\nabla v
            +\xi u(u+1)^{q-2}\nabla w\big), 
\end{align*}
where $m, p, q \in \R$, 
due to the quasilinear structure of 
nonlinearities, 
it is no longer valid to use the transformation $z := \chi v-\xi w$. 
Meanwhile, 
global existence and boundedness of solutions 
have already been 
shown 
in the parabolic--elliptic--elliptic case
under the condition that $p<q$, or $p=q$ 
and $\chi\alpha-\xi\gamma<0$ (\cite{CY-submitted}). 
However, the following question remains: 
\begin{center}
{\it Does the global bounded classical solution converge to\\
a spatially homogeneous steady state?}
\end{center}
The principal purpose of the present paper is to provide an answer 
to this question. 
\medskip

\noindent
{\bf Main result.} 
To achieve the aforementioned purpose, we concentrate on the quasilinear 
parabolic--elliptic--elliptic attraction-repulsion chemotaxis system
%
\begin{align}\label{ARP}
    \begin{cases}
        u_t=\nabla\cdot \big((u+1)^{m-1}\nabla u
            -\chi u(u+1)^{p-2}\nabla v
            +\xi u(u+1)^{q-2}\nabla w\big), 
          \\
        0=\Delta v+\alpha u-\beta v,
          \\
        0=\Delta w+\gamma u-\delta w,
          \\
         \nabla u \cdot \nu|_{\pa \Omega}
        =\nabla v \cdot \nu|_{\pa \Omega}
        =\nabla w \cdot \nu|_{\pa \Omega}=0,
          \\
        u(\cdot, 0)=u_0
    \end{cases}
\end{align}
%
in a bounded domain $\Omega \subset \R^n$ $(n \in \N)$ 
with smooth boundary $\pa\Omega$, 
where $m, p, q \in \R$, $\chi, \xi, \alpha, \beta, \gamma, \delta>0$ 
are constants, 
$\nu$ is the outward normal vector to $\pa\Omega$, 
%
\begin{align}\label{initial}
              u_0 \in C^0(\cl{\Omega}), 
    \quad 
              u_0 \ge 0\ \,{\rm in}\ \cl{\Omega}
    \quad 
              {\rm and}
    \quad 
              u_0 \neq 0.
\end{align}

\medskip
\noindent
The main result of this paper reads as follows. 
%
\begin{thm}\label{stab}
Let $n \in \N$.
Assume that $p, q, \chi, \xi, \alpha, \gamma$ satisfy either 
%
\begin{align}\label{condi1}
   p<q
\end{align}
%
or
%
\begin{align}\label{condi2}
              p=q
   \quad 
              {\it and}
   \quad
              \chi\alpha-\xi\gamma<0, 
\end{align}
%
and that $m, p$ fulfill 
%
\begin{align}\label{condi3}
                   p-m \in [0, 1]
        \ \,\mbox{when}\ 
                   n=1,
        \quad  p-m \in \Big[0, \frac2n\Big]
        \ \,\mbox{when}\ 
                   n \ge 2.
\end{align}
%
Suppose that $u_0$ satisfies \eqref{initial} and 
%
\begin{align}\label{condi4}
         \chi\alpha\|u_0\|_{L^1(\Omega)}^{p-m} 
    \le \frac{1}{2C_{\langle p-m\rangle}},
\end{align}
%
where $C_{\langle p-m\rangle}>0$ is a constant appearing 
in {\rm Lemma~\ref{PSlemma}}. 
Then the solution $(u, v, w)$ of the problem~\eqref{ARP} fulfills 
%
\begin{alignat}{3}
                 &u(\cdot, t) \to \baru 
                 &\ \, {\it in}\ L^\infty(\Omega) 
                 &\quad {\it as}\ t \to \infty  \label{stab u} \\
\intertext{and}
                 &v(\cdot, t) \to \frac{\alpha}{\beta}\baru
                 &\ \, {\it in}\ L^\infty(\Omega)
                 &\quad {\it as}\ t \to \infty  \label{stab v} \\
\intertext{as well as}
                 &w(\cdot, t) \to \frac{\gamma}{\delta}\baru
                 &\ \, {\it in}\ L^\infty(\Omega)
                 &\quad {\it as}\ t \to \infty,  \label{stab w}
\end{alignat}
%
where $\baru := \frac{1}{|\Omega|}\int_\Omega u_0$. 
\end{thm}

\begin{remark}
The conditions \eqref{condi1} and \eqref{condi2} in Theorem~\ref{stab} 
are imposed only to guarantee boundedness. 
In other words, the same conclusion holds by assuming only the conditions 
\eqref{condi3} and \eqref{condi4} when boundedness is known. 
\end{remark}

\noindent
{\bf Strategy and plan of the paper.} 
The strategy for proving stabilization is to construct the energy inequality  
%
\begin{align}\label{energy}
          \frac{d}{dt}\int_\Omega \Phi(u)
          +\ep_0\int_\Omega (u-\baru)^2
     \le 0
\end{align}
%
for some $\ep_0>0$ and some nonnegative functional $\Phi$. 
The key to the derivation of this inequality is to estimate
the diffusion and attraction terms by $\int_\Omega (u-\baru)^2$. 
Specifically, we estimate the former by employing
the Poincar\'e--Sobolev inequality. 
Also, as to the latter, we handle it by taking advantage of 
a favorable structure of the second equation in \eqref{ARP}. 
Thereafter, using a smallness condition for $u_0$, 
we arrive to \eqref{energy}. 

This paper is organized as follows. 
In Section~\ref{Sec2} we state the Poincar\'e--Sobolev inequality 
which will be employed for a term derived from the diffusion one. 
In addition, we give some property of a uniformly continuous function. 
Section~\ref{Sec3} is devoted to the proof of stabilization 
in the problem \eqref{ARP}. 


\section{Preliminaries} \label{Sec2}

In this section we collect two basic facts which will be used later. 
We first recall the Poincar\'e--Sobolev inequality 
which is proved based on 
the Sobolev embedding theorem 
and the Poincar\'e--Wirtinger inequality. 

\begin{lem}\label{PSlemma}
Let $n \in \N$. 
Assume that $\theta \in (-1, 1]$ when $n=1$ and 
that $\theta \in (-1, \frac{2}{n}]$ when $n \ge 2$. 
Then there exists $C_{\langle \theta\rangle}>0$ such that 
%
\begin{align}\label{PS}
               \|\varphi-\cl{\varphi}\|_{L^2(\Omega)}
          \le C_{\langle \theta \rangle}
               \|\nabla\varphi\|_{L^{\frac{2}{\theta+1}}(\Omega)}
\end{align}
%
for all $\varphi \in W^{1, \frac{2}{\theta+1}}(\Omega)$, 
where $\cl{\varphi} := \frac{1}{|\Omega|}\int_\Omega \varphi$. 
\end{lem}

\begin{proof}
We set $\sigma := \frac{2}{\theta+1}>0$ for $\theta>-1$. 
By assumption, we have $\sigma \in [1, \infty)$ when $n=1$. 
Also, in the case $n \ge 2$ we observe from 
the condition $\theta\le \frac{2}{n}$ that $\sigma^* \ge 2$; 
note that $\sigma^*$ is defined as 
$
\frac{1}{\sigma^*}
=\frac{1}{\sigma}-\frac{1}{n}
=\frac{\theta+1}{2}-\frac{1}{n} 
\ \big(\le \frac{1}{2}\big).
$
Therefore we see from the Sobolev embedding theorem that 
there exists $c_1>0$ such that 
\begin{align*}
        \|\varphi-\cl{\varphi}\|_{L^2(\Omega)} 
&\le c_1\|\varphi-\cl{\varphi}\|_{W^{1, \sigma}(\Omega)} \\
&\le c_1\big(\|\varphi-\cl{\varphi}\|_{L^{\sigma}(\Omega)}
                    +\|\nabla \varphi\|_{L^{\sigma}(\Omega)}\big)
\end{align*}
for all $\varphi \in W^{1, \sigma}(\Omega)$. 
Moreover, we can employ the Poincar\'e--Wirtinger inequality
\begin{align*}
\|\varphi-\cl{\varphi}\|_{L^\sigma(\Omega)} \le 
c_2\|\nabla \varphi\|_{L^\sigma(\Omega)}
\end{align*}
for all $\varphi \in W^{1, \sigma}(\Omega)$ with some $c_2>0$, 
because $\sigma \in [1, \infty)$ is assured by the condition 
$\theta \le 1$ in the case $n=1$ and $\theta \le \frac{2}{n} \le 1$ 
in the case $n \ge 2$.  
Combining the above inequalities, we arrive to the conclusion. 
\end{proof}

We next give the following lemma which is called Barbalat's lemma. 
For the proof, see \cite[Lemma~4.2]{SL-1991}, for instance. 
%
\begin{lem}\label{uni}
Assume that $f \colon [0, \infty) \to \R$ is a uniformly continuous 
nonnegative function satisfying 
\begin{align*}
         \int_0^\infty f(t)\,dt < \infty.
\end{align*} 
Then 
\begin{align*}
              \lim_{t \to \infty} f(t)=0. 
\end{align*}
\end{lem}
%


\section{Stabilization}
\label{Sec3}

In this section we assume that $p, q, \chi, \xi, \alpha, \gamma$
fulfill either \eqref{condi1} or \eqref{condi2}, 
and that $u_0$ satisfies \eqref{initial}. 
Then we denote by $(u, v, w)$  the global classical solution 
of the problem~\eqref{ARP} given in \cite{CY-submitted}. 
\medskip

The goal of this section is to deduce stabilization in the problem~\eqref{ARP}. 
To this end, we will establish the key inequality \eqref{DI Phi}, 
which, as will be shown in Lemma~\ref{int u-u_0 finite} below, 
leads to the energy inequality
\begin{align*}
          \frac{d}{dt}\int_\Omega \Phi(u)
          +\ep_0\int_\Omega (u-\baru)^2
     \le 0, 
\end{align*} 
where a constant $\ep_0>0$ and a functional $\Phi$ will be given in later. 
In order to construct the key inequality, we define the functions 
%
\begin{align}
                 V(x,t) &:= v(x,t)-\frac{\alpha}{\beta}\baru,
       \quad 
                 x \in \Omega,\ t>0, \label{Vdef}
\intertext{and}
                 W(x,t) &:= w(x,t)-\frac{\gamma}{\delta}\baru, 
       \quad 
                  x \in \Omega,\ t>0. \label{Wdef}
\end{align}
%
We here note from the second and third equations in \eqref{ARP} that 
$V$ and $W$ satisfy
%
\begin{align}
                 0&=\Delta V+\alpha (u-\baru)-\beta V,  \label{Veq}
\intertext{and}
                 0&=\Delta W+\gamma (u-\baru)-\delta W.  \label{Weq}
\end{align}
%
We first derive some identity to prove 
the key inequality. 
%
\begin{lem}\label{tech}
Let $V$ be the function defined in \eqref{Vdef}. 
Then the following identity holds\/{\rm :}
%
\begin{align}\label{uvV}
                 0=\int_\Omega \nabla u(\cdot, t) \cdot \nabla v(\cdot, t)
                     -\alpha\int_\Omega (u(\cdot, t)-\baru)^2
                     +\beta\int_\Omega (u(\cdot, t)-\baru)V(\cdot, t)
\end{align}
%
for all $t>0$. 
\end{lem}
%

\begin{proof}
Multiplying \eqref{Veq} by $u-\baru$ and integrating it over $\Omega$, we obtain 
%
\begin{align}\label{Veq1}
         0&=\int_\Omega (u-\baru)\Delta V
                +\alpha\int_\Omega (u-\baru)^2
                -\beta\int_\Omega (u-\baru)V.
\end{align}
%
Also, integration by parts entails
%
\begin{align*}
               \int_\Omega (u-\baru)\Delta V
         &= -\int_\Omega \nabla u \cdot \nabla v, 
\end{align*}
%
where we used the identity $\nabla V=\nabla v$. 
This together with \eqref{Veq1} yields \eqref{uvV}. 
\end{proof}

We next establish the relation between 
$\int_\Omega V^2$ and $\int_\Omega (u-\baru)^2$, 
and the one between $\int_\Omega W^2$ and $\int_\Omega (u-\baru)^2$. 
They will play an important role to obtain \eqref{stab v} and \eqref{stab w}. 
%
\begin{lem}\label{L2est V, W}
The functions $V$ and $W$ defined in \eqref{Vdef} and \eqref{Wdef} satisfy
%
\begin{align}
     \int_\Omega V^2(\cdot, t) \le \frac{\alpha^2}{\beta^2}\int_\Omega (u(\cdot, t)-\baru)^2 \label{VL2} 
     \intertext{and}
     \int_\Omega W^2(\cdot, t) \le \frac{\gamma^2}{\delta^2}\int_\Omega (u(\cdot, t)-\baru)^2  \label{WL2}
\end{align}
%
for all $t>0$.
\end{lem}
%

\begin{proof}
Multiplying \eqref{Veq} by $V$ and integrating it over $\Omega$, we see that 
%
\begin{align*}
         0&=\int_\Omega V\Delta V
                +\alpha\int_\Omega (u-\baru)V
                -\beta\int_\Omega V^2 \notag \\
           &=-\int_\Omega |\nabla V|^2
                +\alpha\int_\Omega (u-\baru)V
                -\beta\int_\Omega V^2, 
\end{align*}
%
which implies
%
\begin{align}\label{intV}
               \beta \int_\Omega V^2 \le \alpha\int_\Omega (u-\baru)V. 
\end{align}
%
Here the term on the right-hand side can be estimated 
by Young's inequality as follows:
%
\begin{align*}
             \int_\Omega (u-\baru)V
        \le \frac{\alpha}{2\beta}\int_\Omega (u-\baru)^2
             +\frac{\beta}{2\alpha}\int_\Omega V^2, 
\end{align*}
%
which along with \eqref{intV} ensures that
%
\begin{align*}
                  \beta \int_\Omega V^2 
             \le \frac{\alpha^2}{2\beta}\int_\Omega (u-\baru)^2
             +\frac{\beta}{2}\int_\Omega V^2, 
\end{align*}
%
that is, \eqref{VL2} holds. 
Similarly, we can derive 
\eqref{WL2} via the relation \eqref{Weq}.
\end{proof}

In order to show stabilization in the problem~\eqref{ARP} 
we introduce the function 
%
\begin{align*}
       \Phi(s) := \int_1^s\int_1^\sigma \frac{1}{\eta(\eta+1)^{p-2}}\,d\eta d\sigma, 
    \quad 
       s \ge 0,
\end{align*}
%
where $p \in \R$ is a constant appearing in the attraction term in \eqref{ARP}. 
In the following lemma we establish the desired key inequality. 
%
\begin{lem}\label{DI int phi}
The first component $u$ of the solution to \eqref{ARP} 
satisfies that 
%
\begin{align}\label{DI Phi}
           \frac{d}{dt}\int_\Omega \Phi(u(\cdot, t))
          +\int_\Omega \frac{(u(\cdot, t)+1)^{m-p+1}}{u(\cdot, t)}|\nabla u(\cdot, t)|^2
     \le 2\chi\alpha\int_\Omega (u(\cdot, t)-\baru)^2
\end{align}
%
for all $t>0$. 
\end{lem}
%

\begin{proof}
The first equation in \eqref{ARP} and the identity 
$\Phi''(u)=\frac{1}{u(u+1)^{p-2}}$ 
as well as straightforward calculations imply that 
%
\begin{align}\label{DI Phi1}
      &\frac{d}{dt}\int_\Omega \Phi(u) \notag \\
      &\quad\,
            =\int_\Omega \Phi'(u)
            \nabla\cdot \big(
            (u+1)^{m-1}\nabla u
            -\chi u(u+1)^{p-2}\nabla v
            +\xi u(u+1)^{q-2}\nabla w
            \big) \notag \\
      &\quad\,=-\int_\Omega \Phi''(u)\cdot (u+1)^{m-1}|\nabla u|^2 \notag \\
      &\qquad\ 
           +\chi\int_\Omega \Phi''(u)\cdot u(u+1)^{p-2}\nabla u \cdot \nabla v
           -\xi\int_\Omega \Phi''(u)\cdot u(u+1)^{q-2}\nabla u \cdot \nabla w \notag\\
      &\quad\,
           =-\int_\Omega \frac{(u+1)^{m-p+1}}{u}|\nabla u|^2
           +\chi\int_\Omega \nabla u \cdot \nabla v
           -\xi\int_\Omega (u+1)^{q-p}\nabla u \cdot \nabla w
\end{align}
%
for all $t>0$. 
We now show that the third term on the right-hand side 
can be estimated by zero. 
Taking into account the third equation in \eqref{ARP}, 
we observe that
%
\begin{align}\label{Est I}
      I &:= -\xi\int_\Omega (u+1)^{q-p}\nabla u \cdot \nabla w \notag\\
        &= -\xi\int_\Omega \nabla\Big[\int_0^u (s+1)^{q-p}\,ds\Big] \cdot \nabla w \notag\\
        &= \frac{\xi}{q-p+1}\int_\Omega (u+1)^{q-p+1}\Delta w \notag\\
        &= \frac{\xi}{q-p+1}\int_\Omega (u+1)^{q-p+1}(\delta w-\gamma u) \notag\\
        &= \frac{\xi\delta}{q-p+1}\int_\Omega (u+1)^{q-p+1}w
             -\frac{\xi\gamma}{q-p+1}\int_\Omega u(u+1)^{q-p+1} \notag\\
        &=\frac{\xi\delta}{q-p+1}\int_\Omega (u+1)^{q-p+1}\Big(w+\frac{\gamma}{\delta}\Big)
             -\frac{\xi\gamma}{q-p+1}\int_\Omega (u+1)^{q-p+2}.
\end{align}
%
Again by the third equation in \eqref{ARP}, we have
%
\begin{align*}
      0=\Delta \Big(w+\frac{\gamma}{\delta}\Big)
          +\gamma (u+1)-\delta\Big(w+\frac{\gamma}{\delta}\Big),
\end{align*}
%
which yields that 
%
\begin{align}\label{est w+d/g}
            \Big\|w(\cdot, t)+\frac{\gamma}{\delta}\Big\|_{L^{q-p+2}(\Omega)}
       \le \frac{\gamma}{\delta}\|u(\cdot, t)+1\|_{L^{q-p+2}(\Omega)}
\end{align}
%
for all $t>0$. 
Therefore, applying the H$\ddot{{\rm o}}$lder inequality 
with exponents $\frac{q-p+1}{q-p+2}$ and $\frac{1}{q-p+2}$ to \eqref{Est I} 
and using the estimate \eqref{est w+d/g}, 
we obtain
%
\begin{align*}
   I &\le \frac{\xi\delta}{q-p+1}\Big(\int_\Omega (u+1)^{q-p+2}\Big)^{\frac{q-p+1}{q-p+2}}
             \Big(\int_\Omega \Big(w+\frac{\gamma}{\delta}\Big)^{q-p+2}\Big)^{\frac{1}{q-p+2}} \notag\\
     &\quad\, -\frac{\xi\gamma}{q-p+1}\int_\Omega (u+1)^{q-p+2} \notag\\
     &\le \frac{\xi\delta}{q-p+1}\Big(\int_\Omega (u+1)^{q-p+2}\Big)^{\frac{q-p+1}{q-p+2}}
             \cdot \frac{\gamma}{\delta}\Big(\int_\Omega (u+1)^{q-p+2}\Big)^{\frac{1}{q-p+2}} \notag\\
     &\quad\, -\frac{\xi\gamma}{q-p+1}\int_\Omega (u+1)^{q-p+2} \notag\\
     &=0, 
\end{align*}
%
which along with \eqref{DI Phi1} implies
%
\begin{align}\label{DI Phi2}
      \frac{d}{dt}\int_\Omega \Phi(u)+\int_\Omega \frac{(u+1)^{m-p+1}}{u}|\nabla u|^2
      \le \chi\int_\Omega \nabla u \cdot \nabla v
\end{align}
%
for all $t>0$. 
Also, in view of \eqref{uvV} we see that 
%
\begin{align}\label{uvV2}
                 \chi\int_\Omega \nabla u \cdot \nabla v
              = \chi\alpha\int_\Omega (u-\baru)^2
                  -\chi\beta\int_\Omega (u-\baru)V.
\end{align}
%
Inserting \eqref{uvV2} into \eqref{DI Phi2} entails that 
%
\begin{align}\label{DI Phi3}
      \frac{d}{dt}\int_\Omega \Phi(u)+\int_\Omega \frac{(u+1)^{m-p+1}}{u}|\nabla u|^2
      \le \chi\alpha\int_\Omega (u-\baru)^2
                  -\chi\beta\int_\Omega (u-\baru)V
\end{align}
%
for all $t>0$. 
By means of the Cauchy--Schwarz inequality and 
the estimate \eqref{VL2} we derive
%
\begin{align*}
              -\chi\beta\int_\Omega (u-\baru)V 
      &\le \chi\beta\Big(\int_\Omega (u-\baru)^2\Big)^{\frac{1}{2}}
                           \Big(\int_\Omega V^2\Big)^{\frac{1}{2}} \notag\\
      &\le \chi\alpha\int_\Omega (u-\baru)^2.
\end{align*}
%
This together with \eqref{DI Phi3} proves \eqref{DI Phi}. 
\end{proof}

We next estimate the term containing $|\nabla u|^2$ in \eqref{DI Phi} 
by $\int_\Omega (u-\baru)^2$ in view of Lemma~\ref{PSlemma}. 

%
\begin{lem}\label{Est Diff}
Let $m, p$ satisfy \eqref{condi3}. 
Then the first component $u$ of the solution to \eqref{ARP}
fulfills that 
%
\begin{align}\label{Diff est}
           \int_\Omega \frac{(u(\cdot, t)+1)^{m-p+1}}{u(\cdot, t)}|\nabla u(\cdot, t)|^2
     \ge \frac{1}{C_{\langle p-m\rangle}\|u_0\|_{L^1(\Omega)}^{p-m}}\int_\Omega (u(\cdot, t)-\baru)^2, 
\end{align}
%
for all $t>0$, 
where $C_{\langle p-m\rangle}>0$ is a constant appearing 
in {\rm Lemma~\ref{PSlemma}} with $\theta=p-m$. 
\end{lem}
%

\begin{proof}
We first estimate 
%
\begin{align}\label{Diff est 2}
           \int_\Omega \frac{(u+1)^{m-p+1}}{u}|\nabla u|^2
     \ge \int_\Omega \frac{|\nabla u|^2}{u^{p-m}}.
\end{align}
%
We next estimate the right-hand side of this inequality from below. 
Invoking H$\ddot{{\rm o}}$lder's inequality 
with exponents $\frac{1}{p-m+1} \le 1$ and $\frac{p-m}{p-m+1} \le 1$, 
because $p-m \ge 0$, 
and noting from the first equation in \eqref{ARP} that 
the mass conservation $\int_\Omega u(\cdot, t)=\int_\Omega u_0$ holds 
for all $t>0$, we infer 
%
\begin{align*}
             \int_\Omega |\nabla u|^{\frac{2}{p-m+1}}
     &=   \int_\Omega \Big(\frac{|\nabla u|^2}{u^{p-m}}\Big)^{\frac{1}{p-m+1}} 
                                 \cdot u^{\frac{p-m}{p-m+1}} \notag\\
     &\le \Big(\int_\Omega \frac{|\nabla u|^2}{u^{p-m}}\Big)^{\frac{1}{p-m+1}}
             \Big(\int_\Omega u\Big)^{\frac{p-m}{p-m+1}} \notag\\
     &= \Big(\int_\Omega \frac{|\nabla u|^2}{u^{p-m}}\Big)^{\frac{1}{p-m+1}}
             \Big(\int_\Omega u_0\Big)^{\frac{p-m}{p-m+1}}, 
\end{align*}
%
which means that 
%
\begin{align}\label{Diff est 4}
             \int_\Omega \frac{|\nabla u|^2}{u^{p-m}}
      \ge \Big(\int_\Omega u_0\Big)^{-(p-m)}
            \Big(\int_\Omega |\nabla u|^{\frac{2}{p-m+1}}\Big)^{p-m+1}.
\end{align}
%
Here we can apply Lemma~\ref{PSlemma} with 
$\theta=p-m$ to the term containing $|\nabla u|^{\frac{2}{p-m+1}}$. 
Indeed, we see from \eqref{condi3} that 
$p-m$ satisfies the assumption of Lemma~\ref{PSlemma}. 
Thus, employing the inequality \eqref{PS} with $\theta=p-m$, 
we can find a constant $C_{\langle p-m\rangle}>0$ such that
%
\begin{align}\label{Diff est 5}
            \Big(\int_\Omega |\nabla u|^{\frac{2}{p-m+1}}\Big)^{p-m+1}
      \ge \frac{1}{C_{\langle p-m\rangle}}\int_\Omega (u-\baru)^2. 
\end{align}
%
Collecting \eqref{Diff est 4} and \eqref{Diff est 5} in \eqref{Diff est 2}, 
we establish \eqref{Diff est}. 
\end{proof}

We finally derive an energy inequality which implies 
boundedness of the integral of $\|u(\cdot, t)-\baru\|_{L^2(\Omega)}^2$ 
with respect to $t$ over $(0, \infty)$.
%
\begin{lem}\label{int u-u_0 finite}
Let $m, p$ fulfill \eqref{condi3} and 
let $C_{\langle p-m\rangle}>0$ be a constant 
as in {\rm Lemma~\ref{PSlemma}} with $\theta=p-m$. 
Then the first component $u$ of the solution to \eqref{ARP} 
satisfies that 
%
\begin{align}\label{est int Phi sum}
           \frac{d}{dt}\int_\Omega \Phi(u(\cdot, t))
          +\Big[\frac{1}{C_{\langle p-m\rangle}\|u_0\|_{L^1(\Omega)}^{p-m}}
                   -2\chi\alpha\Big]\int_\Omega (u(\cdot, t)-\baru)^2
     \le 0
\end{align}
%
for all $t>0$. 
In particular, if $u_0$ meets \eqref{condi4}, then 
%
\begin{align}\label{int u-u_0 fin.}
         \int_0^\infty \int_\Omega (u-\baru)^2<\infty.
\end{align}
%
\end{lem}
%

\begin{proof}
Due to Lemmas~\ref{DI int phi} and~\ref{Est Diff}, 
we have
%
\begin{align*}
           \frac{d}{dt}\int_\Omega \Phi(u)
          +\frac{1}{C_{\langle p-m\rangle}\|u_0\|_{L^1(\Omega)}^{p-m}}
             \int_\Omega (u-\baru)^2
     \le 2\chi\alpha\int_\Omega (u-\baru)^2 
\end{align*}
%
for all $t>0$, which entails \eqref{est int Phi sum}. 
Also, integrating the inequality \eqref{est int Phi sum} over $(0, t)$, we obtain 
%
\begin{align*}
           \int_\Omega \Phi(u(\cdot, t))
          +\Big[\frac{1}{C_{\langle p-m\rangle}\|u_0\|_{L^1(\Omega)}^{p-m}}
                   -2\chi\alpha\Big]\int_0^t\int_\Omega (u-\baru)^2
     \le \int_\Omega \Phi(u_0)
\end{align*}
%
for all $t>0$, 
which in conjunction with the positivity of $\Phi$ enables us to see that 
%
\begin{align*}
           \Big[\frac{1}{C_{\langle p-m\rangle}\|u_0\|_{L^1(\Omega)}^{p-m}}
                   -2\chi\alpha\Big]\int_0^t\int_\Omega (u-\baru)^2
     \le \int_\Omega \Phi(u_0).
\end{align*}
%
In view of \eqref{condi4}, 
taking the limit $t \to \infty$, 
we derive \eqref{int u-u_0 fin.}. 
\end{proof}

We are now in a position to complete the proof of Theorem~\ref{stab}. 

\begin{prthstab} 
We first prove an $L^\infty$-convergence of $u$. 
Since the first component $u$ of the solution to \eqref{ARP} 
has the property 
$\sup_{t>0}\|u(\cdot, t)\|_{L^\infty(\Omega)}<\infty$, 
standard parabolic regularity theory (\cite{LSU-1968}) yields 
that there exist $\sigma \in (0, 1)$ and $c_1>0$ such that
%
\begin{align}\label{reg}
        \|u\|_{C^{2+\sigma, 1+\frac{\sigma}{2}}(\cl{\Omega} \times [0, \infty))}
        \le c_1, 
\end{align}
%
which implies that the nonnegative function 
$t \mapsto \|u(\cdot, t)-\baru\|_{L^2(\Omega)}^2$
is uniformly continuous in $[0, \infty)$. 
Hence, taking into account \eqref{int u-u_0 fin.}, 
we infer from Lemma~\ref{uni} that 
%
\begin{align}\label{L2conv}
           \|u(\cdot, t)-\baru\|_{L^2(\Omega)} \to 0\quad {\rm as}\ t \to \infty.
\end{align}
%
Also, by the Gagliardo--Nirenberg inequality we can find $c_2>0$ such that 
%
\begin{align}\label{GN}
                 \|u(\cdot, t)-\baru\|_{L^\infty(\Omega)}
      \le c_2\|u(\cdot, t)-\baru\|_{W^{1,\infty}(\Omega)}^{\frac{n}{n+2}}
                 \|u(\cdot, t)-\baru\|_{L^2(\Omega)}^{\frac{2}{n+2}}.
\end{align}
%
Noting from \eqref{reg} that 
$\|u(\cdot, t)-\baru\|_{W^{1,\infty}(\Omega)} 
\le c_3 := c_1+\baru$, 
we obtain that \eqref{L2conv} and \eqref{GN} assert
%
\begin{align*}
            \|u(\cdot, t)-\baru\|_{L^\infty(\Omega)} \to 0\quad {\rm as}\ t \to \infty,
\end{align*}
%
which means that \eqref{stab u} holds. 
We next show $L^\infty$-convergences of $v$ and $w$. 
In view of Lemma~\ref{L2est V, W}, 
we have from \eqref{L2conv} that
%
\begin{align*}
                \Big\|v(\cdot, t)-\frac{\alpha}{\beta}\baru\Big\|_{L^2(\Omega)}^2
        &\le \frac{\alpha^2}{\beta^2}
                \|u(\cdot, t)-\baru\|_{L^2(\Omega)}^2 \to 0\quad {\rm as}\ t \to \infty 
\intertext{and}
                \Big\|w(\cdot, t)-\frac{\gamma}{\delta}\baru\Big\|_{L^2(\Omega)}^2
        &\le \frac{\gamma^2}{\delta^2}
                \|u(\cdot, t)-\baru\|_{L^2(\Omega)}^2 \to 0\quad {\rm as}\ t \to \infty. 
\end{align*}
%
Therefore, by a similar argument as in the derivation of 
the $L^\infty$-convergence of $u$, 
we can arrive to \eqref{stab v} and \eqref{stab w}. \qed
\end{prthstab}

%
%

\begin{remark}
In Theorem~\ref{stab} 
we can remove the upper bounds for $p-m$ in \eqref{condi3} 
by modifying a condition for $u_0$. 
Indeed, without the upper bounds for $p-m$, 
stabilization in \eqref{ARP} holds in the following form: 

Let $n \in \N$.
Assume that $p, q, \chi, \xi, \alpha, \gamma$ satisfy either 
\eqref{condi1} or \eqref{condi2}, 
and that $m, p$ fulfill $p-m \ge 0$. 
Suppose that $u_0$ satisfies \eqref{initial}. 
Then one can find $\ep_0>0$ such that if 
\begin{align}\label{smallness}
\|u_0\|_{L^\infty(\Omega)} \le \ep 
\end{align}
for all $\ep \in (0, \ep_0)$, 
then the solution $(u, v, w)$ of the problem~\eqref{ARP} fulfills 
\eqref{stab u}--\eqref{stab w}. 

We briefly give the proof. 
By the proof of \cite[Lemma~3.3]{CY-submitted}, 
we see that there exists $\sigma>n$ such that 
\begin{align}\label{Lsigest}
\|u(\cdot, t)\|_{L^\sigma (\Omega)} 
\le \max\{C_\sigma(\|u_0\|_{L^1(\Omega)}),\ \|u_0\|_{L^\sigma(\Omega)}\}
\end{align}
for all $t>0$ with 
$C_\sigma(\|u_0\|_{L^1(\Omega)})
=c_1\|u_0\|_{L^1(\Omega)}(1+\|u_0\|_{L^1(\Omega)}^{c_2})$, 
where $c_1,c_2>0$ are constants independent of $u_0$. 
Also, in view of the proof of \cite[Lemma~A.1]{TW-2012}
we have 
\begin{align*}
\|u(\cdot, t)\|_{L^\infty(\Omega)} 
\le c_3\sup_{t>0}\|u(\cdot, t)\|_{L^\sigma (\Omega)} ^{c_4}
     +\|u_0\|_{L^\infty(\Omega)}
\end{align*}
for all $t>0$, where $c_3, c_4>0$ are constants independent of $u_0$. 
This along with \eqref{Lsigest} derives
\begin{align*}
\|u(\cdot, t)\|_{L^\infty(\Omega)} 
&\le c_5\max\Big\{\|u_0\|_{L^\infty(\Omega)}^{c_4}
(1+\|u_0\|_{L^\infty(\Omega)}^{c_2})^{c_4},\ \|u_0\|_{L^\infty(\Omega)}\Big\}+\|u_0\|_{L^\infty(\Omega)}\\
&=:u_{\rm max} 
\end{align*}
for all $t>0$, where $c_5>0$ 
are constants independent of $u_0$. 
Thus we can find $\ep_0>0$ such that 
if $\|u_0\|_{L^\infty(\Omega)} \le \ep$ for all $\ep \in (0, \ep_0)$, 
then $u_{\rm max} \le
\big(\frac{1}{2\chi\alpha C_{\rm PW}^2}\big)^{\frac{1}{p-m}}$, 
where $C_{\rm PW}>0$ is a constant appearing 
in the Poincar\'e--Wirtinger inequality
\begin{align*}
\|\varphi-\cl{\varphi}\|_{L^2(\Omega)} 
\le C_{\rm PW}\|\nabla \varphi\|_{L^2(\Omega)}\quad 
\mbox{for all}\ \varphi \in W^{1,2}(\Omega). 
\end{align*}
Employing this inequality and 
noting that $p-m \ge 0$, we obtain that
\begin{align*}
        \|u(\cdot,t)-\baru\|_{L^2(\Omega)}^2 
&\le C_{\rm PW}^2\|\nabla u\|_{L^2(\Omega)}^2
\le C_{\rm PW}^2u_{\rm max}^{p-m}
                             \int_\Omega \frac{|\nabla u|^2}{u^{p-m}}
\end{align*}
for all $t>0$, which in conjunction with \eqref{Diff est 2} entails that 
\begin{align*}
\int_\Omega \frac{(u+1)^{m-p+1}}{u}|\nabla u|^2
\ge \frac{1}{C_{\rm PW}^2u_{\rm max}^{p-m}}\int_\Omega (u-\baru)^2. 
\end{align*}
This along with Lemma~\ref{DI int phi} implies that 
\begin{align*}
           \frac{d}{dt}\int_\Omega \Phi(u)
          +\Big[\frac{1}{C_{\rm PW}^2u_{\rm max}^{p-m}}
                   -2\chi\alpha\Big]\int_\Omega (u-\baru)^2
     \le 0
\end{align*}
for all $t>0$. 
Noting that if $\|u_0\|_{L^\infty(\Omega)} \le \ep$ for all $\ep \in (0, \ep_0)$, 
then $u_{\rm max} \le
\big(\frac{1}{2\chi\alpha C_{\rm PW}^2}\big)^{\frac{1}{p-m}}$, 
we can verify $\frac{1}{C_{\rm PW}^2u_{\rm max}^{p-m}}-2\chi\alpha\ge0$ 
when \eqref{smallness} holds. 
Therefore we can derive the conclusion by an argument similar to that 
of the proof of Theorem~\ref{stab}. 
\end{remark}

\section*{Acknowledgments}
The author would like to thank Professor~Tomomi~Yokota 
for his encouragement and helpful comments on the manuscript. 


\begin{thebibliography}{10}

\bibitem{AT-2021}
G.~Arumugam and J.~Tyagi.
\newblock {K}eller--{S}egel chemotaxis models: A review.
\newblock {\em Acta Appl.\ Math.}, {\bf 171}(6):82 pp., 2021.

\bibitem{BBTW-2015}
N.~Bellomo, A.~Bellouquid, Y.~Tao, and M.~Winkler.
\newblock Toward a mathematical theory of {K}eller--{S}egel models of pattern
  formation in biological tissues.
\newblock {\em Math.\ Models Methods Appl.\ Sci.}, {\bf 25}(9):1663--1763,
  2015.

\bibitem{BL-2013}
A.~Blanchet and P.~Lauren\c{c}ot.
\newblock The parabolic--parabolic {K}eller--{S}egel system with critical
  diffusion as a gradient flow in {$\Bbb R^d,\ d\ge3$}.
\newblock {\em Comm.\ Partial Differential Equations}, {\bf 38}(4):658--686,
  2013.

\bibitem{Cao-2015}
X.~Cao.
\newblock Global bounded solutions of the higher-dimensional {K}eller--{S}egel
  system under smallness conditions in optimal spaces.
\newblock {\em Discrete Contin.\ Dyn.\ Syst.}, {\bf 35}(5):1891--1904, 2015.

\bibitem{CY-submitted}
Y.~Chiyo and T.~Yokota.
\newblock Boundedness and finite-time blow-up in a quasilinear
  parabolic--elliptic--elliptic attraction-repulsion chemotaxis system.
\newblock {\em Z. Angew.\ Math.\ Phys.}, {\bf xx}:Paper No.\ xx, xx pp., xxxx.

\bibitem{CS-2012}
T.~Cie\'{s}lak and C.~Stinner.
\newblock Finite-time blowup and global-in-time unbounded solutions to a
  parabolic--parabolic quasilinear {K}eller--{S}egel system in higher
  dimensions.
\newblock {\em J. Differential Equations}, {\bf 252}(10):5832--5851, 2012.

\bibitem{CS-2014}
T.~Cie\'{s}lak and C.~Stinner.
\newblock Finite-time blowup in a supercritical quasilinear
  parabolic--parabolic {K}eller--{S}egel system in dimension $2$.
\newblock {\em Acta Appl.\ Math.}, {\bf 129}:135--146, 2014.

\bibitem{CW-2017}
T.~Cie\'{s}lak and M.~Winkler.
\newblock Stabilization in a higher-dimensional quasilinear {K}eller--{S}egel
  system with exponentially decaying diffusivity and subcritical sensitivity.
\newblock {\em Nonlinear Anal.}, {\bf 159}:129--144, 2017.

\bibitem{HP-2009}
T.~Hillen and K.~J. Painter.
\newblock A user's guide to {PDE} models for chemotaxis.
\newblock {\em J. Math.\ Biol.}, {\bf 58}(1--2):183--217, 2009.

\bibitem{I-2013}
S.~Ishida.
\newblock {$L^\infty$}-decay property for quasilinear degenerate
  parabolic--elliptic {K}eller--{S}egel systems.
\newblock {\em Discrete Contin.\ Dyn.\ Syst.}, (Dynamical systems, differential
  equations and applications. 9th AIMS Conference.\ Suppl.):335--344, 2013.

\bibitem{ISY-2014}
S.~Ishida, K.~Seki, and T.~Yokota.
\newblock Boundedness in quasilinear {K}eller--{S}egel systems of
  parabolic--parabolic type on non-convex bounded domains.
\newblock {\em J. Differential Equations}, {\bf 256}(8):2993--3010, 2014.

\bibitem{IY-2012}
S.~Ishida and T.~Yokota.
\newblock Global existence of weak solutions to quasilinear degenerate
  {K}eller--{S}egel systems of parabolic--parabolic type.
\newblock {\em J. Differential Equations}, {\bf 252}(2):1421--1440, 2012.

\bibitem{IY-2020}
S.~Ishida and T.~Yokota.
\newblock Boundedness in a quasilinear fully parabolic {K}eller--{S}egel system
  via maximal {S}obolev regularity.
\newblock {\em Discrete Contin.\ Dyn.\ Syst.\ Ser.\ S}, {\bf 13}(2):212--232,
  2020.

\bibitem{KS-1970}
E.~F. Keller and L.~A. Segel.
\newblock Initiation of slime mold aggregation viewed as an instability.
\newblock {\em J. Theoret.\ Biol.}, {\bf 26}(3):399--415, 1970.

\bibitem{LSU-1968}
O.~A. Lady\v{z}enskaja, V.~A. Solonnikov, and N.~N. Ural'ceva.
\newblock {\em Linear and {Q}uasilinear {E}quations of {P}arabolic {T}ype}.
\newblock AMS, Providence, 1968.

\bibitem{LM-2017}
P.~Lauren\c{c}ot and N.~Mizoguchi.
\newblock Finite time blowup for the parabolic--parabolic {K}eller--{S}egel
  system with critical diffusion.
\newblock {\em Ann.\ Inst.\ H. Poincar\'{e} Anal.\ Non Lin\'{e}aire}, {\bf
  34}(1):197--220, 2017.

\bibitem{LX-2015}
X.~Li and Z.~Xiang.
\newblock Boundedness in quasilinear {K}eller--{S}egel equations with nonlinear
  sensitivity and logistic source.
\newblock {\em Discrete Contin.\ Dyn.\ Syst.}, {\bf 35}(8):3503--3531, 2015.

\bibitem{LLM-2015}
Y.~Li, K.~Lin, and C.~Mu.
\newblock Asymptotic behavior for small mass in an attraction-repulsion
  chemotaxis system.
\newblock {\em Electron.\ J. Differential Equations}, {\bf 2015}(146):13 pp.,
  2015.

\bibitem{LMW-2015}
K.~Lin, C.~Mu, and L.~Wang.
\newblock Large-time behavior of an attraction-repulsion chemotaxis system.
\newblock {\em J. Math.\ Anal.\ Appl.}, {\bf 426}(1):105--124, 2015.

\bibitem{LCEM-2003}
M.~Luca, A.~Chavez-Ross, L.~Edelstein-Keshet, and A.~Mogliner.
\newblock Chemotactic signalling, microglia, and {A}lzheimer's disease senile
  plague: Is there a connection?
\newblock {\em Bull.\ Math.\ Biol.}, {\bf 65}:673--730, 2003.

\bibitem{M-2017}
Y.~Mimura.
\newblock The variational formulation of the fully parabolic {K}eller--{S}egel
  system with degenerate diffusion.
\newblock {\em J. Differential Equations}, {\bf 263}(2):1477--1521, 2017.

\bibitem{Mizu-2018}
M.~Mizukami.
\newblock The fast signal diffusion limit in a chemotaxis system with strong
  signal sensitivity.
\newblock {\em Math.\ Nachr.}, {\bf 291}(8-9):1342--1355, 2018.

\bibitem{Mizu-2019}
M.~Mizukami.
\newblock The fast signal diffusion limit in a {K}eller--{S}egel system.
\newblock {\em J. Math.\ Anal.\ Appl.}, {\bf 472}(2):1313--1330, 2019.

\bibitem{N-1995}
T.~Nagai.
\newblock Blow-up of radially symmetric solutions to a chemotaxis system.
\newblock {\em Adv.\ Math.\ Sci.\ Appl.}, {\bf 5}(2):581--601, 1995.

\bibitem{N-2001}
T.~Nagai.
\newblock Blowup of nonradial solutions to parabolic--elliptic systems modeling
  chemotaxis in two-dimensional domains.
\newblock {\em J. Inequal.\ Appl.}, {\bf 6}(1):37--55, 2001.

\bibitem{PH-2002}
K.~J. Painter and T.~Hillen.
\newblock Volume-filling and quorum-sensing in models for chemosensitive
  movement.
\newblock {\em Can.\ Appl.\ Math.\ Q.}, {\bf 10}(4):501--543, 2002.

\bibitem{SL-1991}
J.~J. Slotine and W.~Li.
\newblock {\em Applied Nonlinear Control}.
\newblock Prentice-Hall, Englewood Cliffs, N. J., 1991.

\bibitem{SK-2006}
Y.~Sugiyama and H.~Kunii.
\newblock Global existence and decay properties for a degenerate
  {K}eller--{S}egel model with a power factor in drift term.
\newblock {\em J. Differential Equations}, {\bf 227}(1):333--364, 2006.

\bibitem{TW-2013}
Y.~Tao and Z-A. Wang.
\newblock Competing effects of attraction vs.\ repulsion in chemotaxis.
\newblock {\em Math.\ Models Methods Appl.\ Sci.}, {\bf 23}(1):1--36, 2013.

\bibitem{TW-2012}
Y.~Tao and M.~Winkler.
\newblock Boundedness in a quasilinear parabolic--parabolic {K}eller--{S}egel
  system with subcritical sensitivity.
\newblock {\em J. Differential Equations}, {\bf 252}(1):692--715, 2012.

\bibitem{TW-2007}
J.~I. Tello and M.~Winkler.
\newblock A chemotaxis system with logistic source.
\newblock {\em Comm.\ Partial Differential Equations}, {\bf 32}(4--6):849--877,
  2007.

\bibitem{W-2010-1}
M.~Winkler.
\newblock Aggregation vs.\ global diffusive behavior in the higher-dimensional
  {K}eller--{S}egel model.
\newblock {\em J. Differential Equations}, {\bf 248}(12):2889--2905, 2010.

\bibitem{W-2013}
M.~Winkler.
\newblock Finite-time blow-up in the higher-dimensional parabolic--parabolic
  {K}eller--{S}egel system.
\newblock {\em J. Math.\ Pures Appl.\ $(9)$}, {\bf 100}(5):748--767, 2013.

\bibitem{W-2018}
M.~Winkler.
\newblock Finite-time blow-up in low-dimensional {K}eller--{S}egel systems with
  logistic-type superlinear degradation.
\newblock {\em Z. Angew.\ Math.\ Phys.}, {\bf 69}(2):Paper No.\ 69, 25 pp.,
  2018.

\end{thebibliography}

\end{document}